\documentclass[11pt,a4paper]{article}

\usepackage{epsf,epsfig,amsfonts,amsgen,amsmath,amstext,amsbsy,amsopn,amsthm,cases,listings
}
\usepackage{ebezier,eepic,enumitem}
\usepackage{color}
\usepackage{multirow}
\usepackage{epstopdf}
\usepackage{graphicx}
\usepackage{pgf,tikz}
\usepackage{mathrsfs}
\usepackage{authblk}
\usetikzlibrary{arrows}

\newtheorem{definition}{Definition} [section]
\newtheorem{theorem}[definition]{Theorem}
\newtheorem{lemma}[definition]{Lemma}

\newtheorem{conjecture}[definition]{Conjecture}
\newtheorem{claim}[definition]{Claim}

\newenvironment{pf}{\noindent{\bf Proof}}

\def\lf{\left\lfloor}

\def\rf{\right\rfloor}

\begin{document}

\title{\bf\Large Structural results for conditionally intersecting families and some applications}

\date{\today}

\author{Xizhi Liu \thanks{Department of Mathematics, Statistics, and Computer Science, University of Illinois, Chicago, IL, 60607 USA.\ Email: xliu246@uic.edu}}

\maketitle

\begin{abstract}
Let $k\ge d\ge 3$ be fixed.
Let $\mathcal{F}$ be a $k$-uniform family on $[n]$.
Then $\mathcal{F}$ is $(d,s)$-conditionally intersecting
if it does not contain $d$ sets with union of size at most $s$ and empty intersection.
Answering a question of Frankl, we present some structural results for families that are $(d,s)$-conditionally intersecting with $s\ge 2k+d-3$,
and families that are $(k,2k)$-conditionally intersecting.
As applications of our structural results we present some new proofs to the upper bounds for
the size of the following $k$-uniform families on $[n]$.
\begin{enumerate}[label=(\alph*)]
\item $(d,2k+d-3)$-conditionally intersecting families with $n\ge 3k^5$.
\item $(k,2k)$-conditionally intersecting families with $n\ge k^2/(k-1)$.
\item Nonintersecting $(3,2k)$-conditionally intersecting families with $n\ge 3k\binom{2k}{k}$.
\end{enumerate}
Our results for $(c)$ confirms a conjecture of Mammoliti and Britz for the case $d=3$.
\end{abstract}

\section{Introduction}
Let $V$ be a set, and let $S, T$ be two subsets of $V$.
Then we use $S - T$ to denote the set $S\setminus T$, and use $\binom{V}{k}$ to denote the collection of all $k$-subsets of $V$.
Let $[n]$ denote the set $\{1,\ldots,n\}$.
A \textit{$d$-cluster} of $k$-sets is a collection of $d$ different $k$-subsets $A_1,\ldots,A_d$ of $[n]$ such that
\begin{align}
|A_1\cup \cdots \cup A_d|\le 2k, \quad {\rm and} \quad |A_1\cap \cdots \cap A_d|=0. \notag
\end{align}
Let $\mathcal{F}$ be a $k$-uniform family on $[n]$.
Then $\mathcal{F}$ is \textit{$(d,s)$-conditionally intersecting} if it does not contain
$d$ sets with union of size at most $s$ and empty intersection. In particular, a family $\mathcal{F}$ is
$(d,2k)$-conditionally intersecting if it does not contain $d$-clusters.
We use $h(n,k,d,s)$ to denote the maximum size of a $(d,s)$-conditionally intersecting family $\mathcal{F}$.

Note that a $k$-uniform family is $(2,2k)$-conditionally intersecting if and only if it is intersecting.
The celebrated Erd\H{o}s-Ko-Rado theorem \cite{EKR} states that $h(n,k,2,2k)\le \binom{n-1}{k-1}$ for all $n\ge 2k$,
and when $n>2k$ equality holds only if $\mathcal{F}$ is a \textit{star}, i.e. a collection of $k$-sets that contain a fixed vertex.
In \cite{frankl1976d-wise}, Frankl showed that the same conclusion holds for $n\ge dk/(d-1)$ when the intersecting condition
is replaced by the \textit{$d$-wise intersecting} condition, i.e. every $d$ sets of $\mathcal{F}$ have nonempty intersection.

\begin{theorem}[Frankl, \cite{frankl1976d-wise}]\label{thm:Frankl-d-wise-intersecting}
Let $k\ge d\ge 3$ be fixed and $n\ge dk/(d-1)$. If $\mathcal{F}\subset\binom{[n]}{k}$ is a $d$-wise intersecting
family, then $|\mathcal{F}|\le\binom{n-1}{k-1}$, with equality only if $\mathcal{F}$ is a star.
\end{theorem}

Later, Frankl and F\"{u}redi \cite{frankl1983new} extended Theorem \ref{thm:Frankl-d-wise-intersecting} and  proved that
$h(n,k,3,2k)\le \binom{n-1}{k-1}$ for all $n\ge k^2+3k$, and they conjectured that the same inequality holds for all $n\ge 3k/2$.
In \cite{mubayi2006erdos}, Mubayi settled their conjecture and posed the following more general conjecture.

\begin{conjecture}[Mubayi, \cite{mubayi2006erdos}]\label{conj:Mubayi-d-cluster}
Let $k\ge d \ge 3$ and $n\ge dk/(d-1)$. Suppose that $\mathcal{F}\subset \binom{[n]}{k}$ is a $(d,2k)$-conditional intersecting family.
Then $|\mathcal{F}|\le\binom{n-1}{k-1}$, with equality only if $\mathcal{F}$ is a star.
\end{conjecture}

Conjecture~\ref{conj:Mubayi-d-cluster} has been intensively studied in the past decade.
Mubayi \cite{mubayiintersection4} proved this conjecture for the case $d=4$ with $n$ sufficiently large.
Later, Mubayi and Ramadurai \cite{mubayi2009set}, and independently, F\"{u}redi and \"{O}zkahya \cite{furedi2011cluster}
settled this conjecture for all $d\ge 3$ with $n$ sufficiently large.
In \cite{chen2009cluster}, Chen, Liu and Wang confirmed this conjecture for the case $d=k$,
and they also showed that $h(n,k,d,(d+1)k/2)\le \binom{n-1}{k-1}$ for all $n\ge dk/(d-1)$.
Very recently, Conjecture \ref{conj:Mubayi-d-cluster}  was completely solved by Currier \cite{GC18}.

In this paper, we consider the structure of conditionally intersecting families,
which is motivated by a structural theorem for $(3,6)$-conditionally intersecting family proved by Frankl \cite{franklstructural}.

\begin{definition}
Let $\mathcal{H}\subset 2^{[n]}$, and let $H\in \mathcal{H}$.
A subset $G\subset H$ is called unique
if there is no other set in $\mathcal{H}$ containing $G$.
\end{definition}

The following result of Bollob\'{a}s  \cite{bollobasgeneralized} gives an upper bound for the size of
a family in which every set has a unique subset.

\begin{theorem}[Bollob\'{a}s, \cite{bollobasgeneralized}]\label{thm:Bollobas-unique-subset}
Suppose that for every member H of the family $\mathcal{H}\subset 2^{[n]}$ the set $G(H)\subset H$
is a unique subset. Then
\begin{align}
\sum_{H\in\mathcal{H}}\frac{1}{\binom{n-|H-G(H)|}{|G(H)|}}\le 1. \notag
\end{align}
\end{theorem}

Frankl \cite{franklstructural} proved the following structural result for  $(3,6)$-conditionally intersecting families.

\begin{theorem}[Frankl, \cite{franklstructural}]\label{thm:Frankl-structure-3-cluster}
Suppose that $\mathcal{F}\subset\binom{[n]}{3}$ is a $(3,6)$-conditionally intersecting family.
Then $\mathcal{F}$ can be partitioned into two families $\mathcal{H}$ and $\mathcal{B}$,
and the ground set $[n]$ can be partitioned into two disjoint subsets $Y$ and $Z$ such that the following statements hold.
\begin{enumerate}[label=(\alph*)]
\item $\mathcal{H}\subset\binom{Y}{3}$ and every set $H\in \mathcal{H}$ contains a unique $2$-subset.
\item $\mathcal{B}\subset\binom{Z}{3}$ and $\mathcal{B}$ is the vertex disjoint union of $|Z|/4$ copies of
complete $3$-graphs on $4$ vertices.
\end{enumerate}
\end{theorem}

First, let us show how to use Theorem \ref{thm:Frankl-structure-3-cluster} to get an upper bound for $|\mathcal{F}|$.
Let $\mathcal{F}\subset\binom{[n]}{3}$ be a $(3,6)$-conditionally intersecting family,
and let $Y,Z, \mathcal{B} \text{ and } \mathcal{H}$ be given by Theorem \ref{thm:Frankl-structure-3-cluster}.
Since every set in $\mathcal{H}$ contains a unique $2$-subset, it follows from Theorem \ref{thm:Bollobas-unique-subset}
that $|\mathcal{H}|\le \binom{|Y|-1}{2}$.
On the other hand, it is easy to see that $|\mathcal{B}|=|Z|$.
Therefore,
\begin{align}
|\mathcal{F}|=|\mathcal{H}|+|\mathcal{F}|\le \binom{|Y|-1}{2}+|Z|\le \binom{n-1}{2}, \notag
\end{align}
and equality holds only if $Z = \emptyset$.

In \cite{franklstructural}, Frankl also asked for a structural result for a $(3,2k)$-conditionally intersecting family
$\mathcal{F}\subset \binom{[n]}{k}$ which can imply the $\binom{n-1}{k-1}$ bound for $|\mathcal{F}|$.
Here we consider a more general question, namely the structures of $(d,2k+d-3)$-conditionally intersecting families for all $k \ge d\ge 3$,
and we obtain the following result.

Let $\mathcal{L}_{k}$ denote the collection of all $k$-graphs on at most $2k$ vertices.

\begin{theorem}\label{thm:d-2k+d-3-structure}
Let $k\ge d\ge 3$ be fixed.
Suppose that $\mathcal{F}\subset \binom{[n]}{k}$ is a $(d,2k+d-3)$-conditionally intersecting family.
Then $\mathcal{F}$ can be partitioned into three families $\mathcal{H}$, $\mathcal{B}$ and $\mathcal{S}$,
and the ground set $[n]$ can be partitioned into two subsets $Y$ and $Z$ such that  the following statements hold.
\begin{enumerate}[label=(\alph*)]
\item $\mathcal{H}\subset \binom{Y}{k}$ and every set $H\in \mathcal{H}$ contains a unique $(k-1)$-subset.
\item $Z$ has a partition $V_1\cup \cdots\cup V_t$ with each $V_i$ of size at most $2k$ such that $\mathcal{B}\subset \bigcup_{i=1}^{t}\binom{V_i}{k}$,
i.e., the family $\mathcal{B}$ is the vertex disjoint union of copies of $k$-graphs in $\mathcal{L}_{k}$
\item $\mathcal{S}\subset \binom{[n]}{k}-\binom{Y}{k}$, and for every set $S\in \mathcal{S}$ and every $V_i\subset Z$
the size of $S\cap V_i$ is either $0$ or at least $d$.
\end{enumerate}
\end{theorem}

Note that the constraint on $|S\cap V_i|$ in $(c)$ for $S\in \mathcal{S}$
and $V_i\subset Z$ implies that the family $\mathcal{S}$ is actually very sparse.
Therefore, the term $|\mathcal{S}|$ contributes very little to $|\mathcal{F}|$.

Our next result gives a structure for $(k,2k)$-intersecting families for all $k\ge 3$.

\begin{theorem}\label{thm:k-2k-structure}
Let $k\ge 3$ be fixed.
Suppose that $\mathcal{F}\subset \binom{[n]}{k}$ is a $(k,2k)$-conditionally intersecting family.
Then $\mathcal{F}$ can be partitioned into two families $\mathcal{H}$ and $\mathcal{B}$,
and the ground set $[n]$ can be partitioned into two subsets $Y$ and $Z$ such that the following statements hold.
\begin{enumerate}[label=(\alph*)]
\item $\mathcal{H}\subset \binom{Y}{k}$ and every set $H\in \mathcal{H}$ contains a unique $(k-1)$-subset.
\item $\mathcal{B}\subset \binom{Z}{k}$ and $\mathcal{B}$ is the vertex disjoint union of $\frac{|Z|}{k+1}$ copies of
complete $k$-graphs on $(k+1)$ vertices.
\end{enumerate}
\end{theorem}

Applying the structural results above we are able to give some new proofs to the following theorems.

\begin{theorem}\label{thm:upper-bound-d-2k+d-3}
Let $k\ge d\ge 3$ be fixed and $n\ge 3k^5$.
Suppose that $\mathcal{F}\subset \binom{[n]}{k}$ is a $(d,2k+d-3)$-conditionally intersecting family.
Then $|\mathcal{F}|\le\binom{n-1}{k-1}$.
\end{theorem}
Note that Theorem \ref{thm:upper-bound-d-2k+d-3} is true for every $n\ge 3k/2$ according to the result in \cite{mubayi2006erdos},
but in our proof we need the assumption that $n\ge 3k^5$ to keep the calculations simple.

\begin{theorem}\label{thm:upper-bound-k-2k}
Let $k\ge 3$ be fixed and $n\ge k^2/(k-1)$.
Suppose that $\mathcal{F}\subset \binom{[n]}{k}$ is a $(k,2k)$-conditionally intersecting family.
Then $|\mathcal{F}|\le\binom{n-1}{k-1}$.
\end{theorem}

\begin{theorem}\label{thm:upper-bound-nonintersecting-3-2k}
Let $k\ge 3$ be fixed and $n\ge 3k\binom{2k}{k}$.
Let $\mathcal{F}\subset \binom{[n]}{k}$ be a family that is $(3,2k)$-conditionally intersecting but not intersecting.
Then $|\mathcal{F}|\le\binom{n-k-1}{k-1}+1$.
\end{theorem}
Theorem \ref{thm:upper-bound-nonintersecting-3-2k} shows that Mammoliti and Britz's conjecture
(Conjecture 4.1 in $\cite{mammoliti2017}$) is true for the case $d=3$.
Note that in \cite{XL19} the author considered  Mammoliti and Britz's conjecture for all $d\ge 3$,
and showed that their conjecture is true for $d=3$, but false for all $d\ge 4$.
However, the method we used here is completely different from the method used in \cite{XL19}.

The remaining part of this paper is organized as follows.
In Section 2, we prove Theorems \ref{thm:d-2k+d-3-structure} and \ref{thm:k-2k-structure}.
In Section 3, we prove Theorems \ref{thm:upper-bound-d-2k+d-3},
\ref{thm:upper-bound-k-2k}, and \ref{thm:upper-bound-nonintersecting-3-2k}.

\section{Structural Results}
Let $\mathcal{F}$ be a $k$-uniform family on $[n]$ and $B\in \mathcal{F}$.
We say $B$ is \textit{bad} if it does not contain any unique $(k-1)$-subset.
Suppose that $B=\{b_1,\ldots, b_k\}$ is a bad set in $\mathcal{F}$, then there exist $k$ distinct sets
$C_1,\ldots, C_k$ in $\mathcal{F}$ such that $B\cap C_i=B-\{b_i\}$ for all $i\in [k]$.
Let $V_B = B\cup C_1\cdots \cup C_k$ and $H_B = \{B,C_1,\ldots, C_k\}$.
First let us prove Theorem \ref{thm:k-2k-structure}.

\begin{proof}[Proof of Theorem \ref{thm:k-2k-structure}]
Suppose that $\mathcal{F}$ is a $(k,2k)$-conditionally intersecting family,
and suppose that $B=\{b_1,\ldots,b_k\}$ is a bad set in $\mathcal{F}$.
Let $C_1,\cdots,C_k,V_{B}, H_{B}$ be defined as above.
Since $|V_{B}| \le  2k$, by assumption we have $C_1\cap \cdots \cap C_k\neq \emptyset$.
It follows that $|V_B|=k+1$ and, hence, the family $H_B$ is a complete $k$-graph on $V_B$.
Let $b_{k+1}$ denote the vertex in $V_B-B$, and let $F\in \mathcal{F}-H_B$. Then we claim that $F\cap V_B=\emptyset$.
Indeed, suppose that $F\cap V_B\neq\emptyset$.
We may assume that $F\cap V_B=\{b_1,\ldots,b_{\ell}\}$ for some $\ell\in [k-1]$.
Now, rename the edges in $H_B$ as $B_i=V_B-b_i$ for all $i\in[k+1]$.
Since $|F\cup B_1\cup \cdots \cup B_{k-1}| \le 2k$ and $F\cap B_1\cap \cdots \cap B_{k-1}=\emptyset$,
the $k$ sets $F,B_1,\ldots,B_{k-1}$ form a $k$-cluster in $\mathcal{F}$, a contradiction.
Therefore, $F\cap V_B=\emptyset$.
To finish the proof we just let $\mathcal{B}$ be the collection of all bad sets in $\mathcal{F}$,
and let $\mathcal{H}=\mathcal{F}-\mathcal{B}$.
\end{proof}

Before proving Theorem \ref{thm:d-2k+d-3-structure} let us present a useful lemma.
Let $s = 2k +d -3$.

\begin{lemma}\label{lemma:d-s-size-intersection}
Suppose that $\mathcal{F}$ is a $(d,s)$-conditionally intersecting family and $B$ is a bad set in $\mathcal{F}$.
Then for every $F\in \mathcal{F}$ either $|F\cap V_B|=0$ or $|F\cap V_B|\ge d$.
\end{lemma}
\begin{proof}
Let $B$ is a bad set in $\mathcal{F}$ and let $V_B$ be the set as we defined before.
Suppose that $F\in \mathcal{F}$ has nonempty intersection with $V_B$.
It suffices to show that $|F\cap V_B|\ge d$.
For contradiction, suppose that $|F\cap B| = x$, $|F\cap (V_{B}-B)| = y$ and $x+y\le d-1$.
Suppose that $F\cap B=\{b_{m_1},\ldots,b_{m_x}\}$ and $F\cap (V_{B}-B) = \{c_{n_1},\ldots,c_{n_y}\}$.

If $x=d-1$, then $y = 0$ and, hence,  the $d$ sets $F,C_{m_1},\ldots,C_{m_{d-1}}$ satisfy
 $|F\cup C_{m_1}\cup \cdots \cup C_{m_{d-1}}|\le 2k$
and $F\cap C_{m_1}\cap \cdots \cap C_{m_{d-1}}=\emptyset$, a contradiction.
If $x=d-2$, then the $d$ sets $F,B,C_{m_1},\ldots,C_{m_{d-2}}$ satisfy
$|F\cup B\cup C_{m_1}\cup \cdots \cup C_{m_{d-2}}|\le 2k$
and $F\cap B\cap C_{m_1}\cap \cdots \cap C_{m_{d-2}}=\emptyset$, a contradiction.
Therefore, we may assume that $x\le d-3$.
Let $p=d-(x+2)$.
Choose $p$ sets $C_{q_1},\ldots,C_{q_p}$ from $\{C_1,\ldots,C_k\}-\{C_{m_1},\ldots,C_{m_{x}}\}$.
Then the $d$ sets $F,B,C_{m_1},\ldots,C_{m_x},C_{q_1},\ldots,C_{q_p}$ satisfy
$|F\cup B\cup C_{m_1}\cup \cdots \cup C_{m_x}\cup C_{q_1}\cup \cdots \cup C_{q_p}|\le 2k+p$
and $F\cap B\cap C_{m_1}\cap \cdots \cap C_{m_x}\cap C_{q_1}\cap \cdots \cap C_{q_p}=\emptyset$.
By assumption we have $2k+p\ge s$ and, hence, $x=0$ and $y\ge 1$.

Let $p'=d-(y+2)$, and choose $p'$ sets $C_{q_1},\ldots,C_{q_{p'}}$ from $\{C_1,\ldots,C_k\}-\{C_{n_1},\ldots,C_{n_{y}}\}$.
Then the $d$ sets $F,B,C_{n_1},\ldots,C_{n_y},C_{q_1},\ldots,C_{q_{p'}}$ satisfy
$|F\cup B\cup C_{n_1}\cup \cdots \cup C_{n_y}\cup C_{q_1}\cup \cdots \cup C_{q_{p'}}|\le 2k+p'\le s$
and $F\cap B\cap C_{n_1}\cap \cdots \cap C_{n_y}\cap C_{q_1}\cap \cdots \cap C_{q_{p'}}=\emptyset$, a contradiction.
Therefore, we have $|F\cap V_b|\ge d$.
\end{proof}

Now we are ready to prove Theorem \ref{thm:d-2k+d-3-structure}.

\begin{proof}[Proof of Theorem \ref{thm:d-2k+d-3-structure}]
Let $\mathcal{F}$ be a $(d,s)$-conditionally intersecting family.
Choose a collection of bad sets $\{B_1,\ldots, B_t\}$ for some $t$ from $\mathcal{F}$
such that the sets $V_{B_1},\ldots, V_{B_t}$ are pairwise disjoint,
and any other bad set in $\mathcal{F}$ has nonempty intersection with $V_{B_i}$ for some $i\in [t]$.
Note that this can be done by greedy choosing each $B_i$ from $\mathcal{F}$ such that $B_i$ is disjoint from $\bigcup_{j<i}V_{B_j}$,
and by Lemma \ref{lemma:d-s-size-intersection} the set $V_{B_i}$ is also disjoint from $\bigcup_{j<i}V_{B_j}$.

Now let $V_i=V_{B_i}$ and $H_i=H_{B_i}$ for $i \in [t]$.
Let $Z=\bigcup_{i\in[t]}V_i$ and $Y=[n]-Z$.
Let $\mathcal{B}=\bigcup_{i\in[t]}H_i$, $\mathcal{H}=\mathcal{F}\cap \binom{Y}{k}$ and
$\mathcal{S}=\mathcal{F}-\mathcal{B}-\mathcal{H}$.
Suppose that $S\in \mathcal{S}$. Then by Lemma \ref{lemma:d-s-size-intersection},
either $|S\cap V_i|=0$ or $|S\cap V_i|\ge d$ for every $i\in[t]$,
and this completes the proof of Theorem \ref{thm:d-2k+d-3-structure}.
\end{proof}

\section{Applications}
In this section we show some applications of Theorems \ref{thm:d-2k+d-3-structure} and \ref{thm:k-2k-structure}
by giving new proofs to Theorems \ref{thm:upper-bound-d-2k+d-3}, \ref{thm:upper-bound-k-2k}, and \ref{thm:upper-bound-nonintersecting-3-2k}.
First let us prove Theorem \ref{thm:upper-bound-k-2k}.

\begin{proof}[Proof of Theorem \ref{thm:upper-bound-k-2k}]
Suppose that $\mathcal{F}$ is a $(k,2k)$-conditionally intersecting family on $[n]$.
Let $Y,Z,\mathcal{B} \text{ and }\mathcal{H}$ be given by Theorem \ref{thm:k-2k-structure}.
By Theorem \ref{thm:Bollobas-unique-subset}, $\mathcal{H}\le \binom{|Y|-1}{k-1}$.
On the other hand, it is easy to see that $|\mathcal{B}| = (k+1) \times |Z|/(k+1) = |Z|$.
Therefore, $|\mathcal{F}|=|\mathcal{H}|+|\mathcal{B}|\le \binom{|Y|-1}{k-1}+|Z|\le \binom{n-1}{k-1}$,
and equality holds only if $Z = \emptyset$.
\end{proof}

Now we apply Theorem \ref{thm:d-2k+d-3-structure} to prove Theorem \ref{thm:upper-bound-d-2k+d-3}.

\begin{proof}[Proof of Theorem \ref{thm:upper-bound-d-2k+d-3}]
Let $\mathcal{F}$ be a $(d,2k+d-3)$-conditionally intersecting family on  $n\ge 3k^5$ vertices.
Let $Y,Z,\mathcal{B},\mathcal{H} \text{ and }\mathcal{S}$ be given by Theorem \ref{thm:d-2k+d-3-structure}.
Let $v_i = |V_i|$ for $i\in[t]$.
Let $Y_0=Y$ and $Y_i=Y_{i-1}\cup V_i$ for $i\in [t]$ and let $y_i = |Y_i|$ for $0 \le i \le t$.
Define $\mathcal{H}_i=\mathcal{F}\cap \binom{Y_i}{k}$ and let $h_i = |\mathcal{H}_i|$.
By Lemma \ref{lemma:d-s-size-intersection},
every set $H\in\mathcal{H}_i$ is either disjoint from $V_i$ or has an intersection of size at least $d$ with $V_i$.
Therefore, $|\mathcal{H}_i|\le |\mathcal{H}_{i-1}|+\sum_{\ell=d}^{k}\binom{v_i}{\ell}\binom{y_{i-1}}{k-\ell}$.
Inductively, we obtain
\begin{align}
|\mathcal{F}|
\le |\mathcal{H}|+\sum_{i=0}^{t-1}\sum_{\ell=d}^{k}\binom{v_{i+1}}{\ell}\binom{y_{i}}{k-\ell}\le \binom{y_0-1}{k-1}+\sum_{i=0}^{t-1}\sum_{\ell=d}^{k}\binom{2k}{\ell}\binom{n-k-1}{k-\ell}. \notag
\end{align}
Since $\binom{2k}{\ell}\binom{n-k-1}{k-\ell}\ge \binom{2k}{\ell+1}\binom{n-k-1}{k-\ell-1}$, we obtain
\begin{align}
|\mathcal{F}|
& \le \binom{y_0-1}{k-1}+\sum_{i=0}^{t-1}(k-d)\binom{2k}{d}\binom{n-k-1}{k-d} \notag\\
&\le \binom{y_0-1}{k-1}+(k-d)\binom{2k}{d}\binom{n-k-1}{k-d}\frac{n-y_0}{k+1} \notag\\
&\le \binom{y_0-1}{k-1}+\binom{2k}{3}\binom{n-k-1}{k-3}(n-y_0). \notag
\end{align}
Now let $\delta=\left({2\binom{2k}{3}}\right)^{-1}$.
If $n-y_0\le \delta n$, then
\begin{align}
|\mathcal{F}|< \binom{n-1}{k-1}-k\binom{n-k-1}{k-2}+\frac{n}{2}\binom{n-k-1}{k-3}<\binom{n-1}{k-1}, \notag
\end{align}
and we are done.
Therefore, we may assume that $y_0\le \left(1-\delta\right)n$.
Then
\begin{align}
|\mathcal{F}|\le \left(1-\frac{1}{4\binom{2k}{3}}\right)\binom{n-1}{k-1}+\binom{n-k-1}{k-3}\frac{n}{2}\le \binom{n-1}{k-1}, \notag
\end{align}
and this completes the proof of Theorem \ref{thm:upper-bound-d-2k+d-3}.
\end{proof}

The remaining part of this section is devoted to prove Theorem \ref{thm:upper-bound-nonintersecting-3-2k}.
We will use the following lemma in our proof.

The \textit{shadow} $\partial\mathcal{H}$ of a family $\mathcal{H}\subset \binom{[n]}{k}$ is defined as follows:
\begin{align}
\partial\mathcal{H} = \left\{G\in \binom{[n]}{k-1}: \exists H\in \mathcal{H} \text{ such that } G\subset H \right\}. \notag
\end{align}

\begin{lemma}\label{lemma:shadow-unique-subset}
Suppose that $\mathcal{H}\subset \binom{[n]}{k}$, and every set $H\in \mathcal{H}$ has a unique $(k-1)$-subset
$G(H)\subset H$. Then
\begin{align}
|\mathcal{H}|\le \frac{n-k+1}{n}|\partial\mathcal{H}|. \notag
\end{align}
\end{lemma}
\begin{proof}
Consider a weight function $\omega(G,H)$ for all pairs $G\subset H\in \mathcal{F}$ with $|G|=k-1$.
For every $G \in \partial \mathcal{H}$ and every $H \in \mathcal{H}$ assign weight 1 to $(G,H)$ if $G=G(H)$ and $(n-k+1)^{-1}$ if $G\neq G(H)$.
Then an easy double counting gives
\begin{align}
\left(1+\frac{k-1}{n-k+1}\right)|\mathcal{H}|=\sum_{(G,H)}\omega(G,H)\le|\partial\mathcal{H}|, \notag
\end{align}
which implies $|\mathcal{H}|\le (n-k+1)|\partial \mathcal{H}| / n$.
\end{proof}

\begin{definition}
Let $\mathcal{F}\subset \binom{[n]}{k}$ and $S \subset [n]$.
Then $\mathcal{F}$ is a {full star} on $S$ if it is the collection of all $k$-subsets of $S$ that contain a fixed vertex $v$,
and $\mathcal{F}$ is a {star} if it is a subfamily of some full star on $S$.
In either case, we call $v$ the {core} of $\mathcal{F}$.
\end{definition}

Now we prove Theorem \ref{thm:upper-bound-nonintersecting-3-2k}.

\begin{proof}[Proof of Theorem \ref{thm:upper-bound-nonintersecting-3-2k}]
Let $n\ge 3k\binom{2k}{k}$ and
let $\mathcal{F}$ be a family on $[n]$ such that $\mathcal{F}$ is $(3,2k)$-conditionally intersecting but not intersecting.
Suppose that $B\in\mathcal{F}$ is a bad set.
Let $V_B,H_B$ be as defined at the beginning of this section and
let $\mathcal{F}'=\mathcal{F}\cap \binom{[n]-V_B}{k}$.
Since $\mathcal{F}'$ is also $(3,2k)$-intersecting, by result in \cite{mubayi2006erdos},
$|\mathcal{F}'|\le \binom{n-|V_B|-1}{k-1}|\le \binom{n-k-2}{k-1}$.
Then by Lemma \ref{lemma:d-s-size-intersection},
\begin{align}
|\mathcal{F}|
& \le |\mathcal{F}'|+\sum_{i=3}^{k}\binom{2k}{i}\binom{n-k-1}{k-i} \notag\\
&\le \binom{n-k-2}{k-1}+k\binom{2k}{3}\binom{n-k-1}{k-3} \notag\\
&=\binom{n-k-1}{k-1}-\left(\binom{n-k-2}{k-2}-k\binom{2k}{3}\binom{n-k-1}{k-3}\right)<\binom{n-k-1}{k-1}+1, \notag
\end{align}
and we are done.
So we may assume that every $F\in\mathcal{F}$ has a unique $(k-1)$-subset $G(F)$.

Since $\mathcal{F}$ is not intersecting, there exist two disjoint sets $A,B \text{ in } \mathcal{F}$.
Assume that $A=\{a_1,\ldots,a_k\}$ and $B=\{b_1,\ldots,b_k\}$.
Let $I=\{a_1,\ldots,a_k,b_1,\ldots,b_k\}$ and let $U=[n]-I$.
For every set $C\subset U$ of size at most $k-1$ define the family $\mathcal{F}(C)$ on $I$ as follows:
\begin{align}
\mathcal{F}(C)= \left\{F- C: F\in \mathcal{F} \text{ and } F\cap U=C \right\}. \notag
\end{align}
For every $i\in\{0,1,...,k\}$ let
\begin{align}
\mathcal{F}_i=\{F\in\mathcal{F}: |F\cap I|=i\}. \notag
\end{align}
First notice that $\mathcal{F}_k=\{A,B\}$,
since any extra edge in $\mathcal{F}_{k}$ together with $A,B$ would form a $3$-cluster in $\mathcal{F}$.
Next, we will prove
\begin{align}\label{inequality-main-goal}
\sum_{i=0}^{\ell}|\mathcal{F}_i|\le\sum_{i=1}^{\ell}\binom{n-2k}{k-i}\binom{k-1}{i-1}.
\end{align}
for all $\ell\in[k]$.
Suppose that $(\ref{inequality-main-goal})$ is true, then by letting $\ell =k$ we obtain
\begin{align}
|\mathcal{F}|=\sum_{i=0}^{k}|\mathcal{F}_i|\le \sum_{i=1}^{k-1}\binom{n-2k}{k-i}\binom{k-1}{i-1}+2=\binom{n-k-1}{k-1}+1, \notag
\end{align}
and this will complete the proof of Theorem \ref{thm:upper-bound-nonintersecting-3-2k}.
One could compare $(\ref{inequality-main-goal})$ with a similar inequality in \cite{mubayi2006erdos}, which is
\begin{align}\label{inequality-dhruv}
|\mathcal{F}|\le \sum_{\ell=1}^{k}\binom{n-tk}{k-\ell}\binom{tk-1}{\ell-1}=\binom{n-1}{k-1},
\end{align}
where $t$ is the maximum number of pairwise disjoint sets in $\mathcal{F}$.
For the case $t=2$, the summand in $(\ref{inequality-dhruv})$ is $\binom{n-2k}{k-\ell}\binom{2k-1}{\ell-1}$,
but the summand  in $(\ref{inequality-main-goal})$ is $\binom{n-2k}{k-\ell}\binom{k-1}{\ell-1}$,
which is smaller when $\ell \ge 2$.

\begin{claim}\label{lemma:F-cap-U-is-unique-subset}
Let $F\in\mathcal{F}_1$. Then the set $F\cap U$ is a unique $(k-1)$-subset of $F$ in $\mathcal{F}$.
\end{claim}
\begin{proof}[Proof of Claim \ref{lemma:F-cap-U-is-unique-subset}]
Without lose of generality, we may assume that $F=\{a_1,f_1,\ldots,f_{k-1}\}$, where $f_1,\ldots,f_{k-1}$ are contained in $U$.
Suppose that there is another edge $F'\in \mathcal{F}$ containing  $\{f_1,\ldots,f_{k-1}\}$.
Then the three sets $A,F,F'$ form a $3$-cluster in $\mathcal{F}$, a contradiction.
Therefore, $F\cap U = \{f_1,\ldots,f_{k-1}\}$ is a unique $(k-1)$-subset of $F$ in $\mathcal{F}$.
\end{proof}

Now we prove $(\ref{inequality-main-goal})$ for $\ell=1$.
Let us consider the family $\mathcal{F}_0\cup \mathcal{F}_1$.
Define
\begin{align}
\mathcal{M}=\left\{G\in \binom{U}{k-1}: \exists F\in \mathcal{F}_0\cup \mathcal{F}_1 \text{ such that } G\subset F \right\}. \notag
\end{align}
By assumption, every set $F\in\mathcal{F}_0\cup \mathcal{F}_1$
has a unique $(k-1)$-subset $G(F)$, and by Claim \ref{lemma:F-cap-U-is-unique-subset},
we may assume that  $G(F) \subset U$.
Let $\mathcal{G}=\left\{G(F): F\in \mathcal{F}_1 \right\}$.
For every set $F_1\in \mathcal{F}_1$, the set $G(F_1)$ cannot be contained in $\partial\mathcal{F}_0$,
since otherwise one could easily find a $3$-cluster.
Therefore, $\mathcal{G}$ and $\partial\mathcal{F}_0$ are disjoint.
Since $|\mathcal{G}|=|\mathcal{F}_1|$, by Lemma \ref{lemma:shadow-unique-subset}, we have
\begin{align}
\frac{|U|}{|U|-k+1}|\mathcal{F}_0|+|\mathcal{F}_1|\le |\mathcal{M}| \le \binom{n-2k}{k-1}, \notag
\end{align}
and hence $|\mathcal{F}_0|+|\mathcal{F}_1|\le \binom{n-2k}{k-1}$.

To prove $(\ref{inequality-main-goal})$ for $\ell \ge 2$,
we need to give an upper bound for $|\mathcal{F}_i|$ for every $2 \le i \le k-1$.
Since $|\mathcal{F}_i|=\sum_{C\in \binom{U}{k-i}}|\mathcal{F}(C)|$,
it suffices to give an upper bound for $|\mathcal{F}(C)|$ for every $C\in\binom{U}{k-i}$.
Unfortunately, the inequality $|\mathcal{F}(C)|\le \binom{k-1}{i-1}$ is not true in general.
So, in our proof, we will build a relationship between $\mathcal{F}_i$ and $\bigcup_{j<i}\mathcal{F}_j$
and then use this relation to prove $(\ref{inequality-main-goal})$.

The basic idea in our proof is showing that if $|\mathcal{F}(C)|$ is bigger than its expected value $\binom{k-1}{k- |C| -1}$,
then there must be many sets $D$ containing $C$ such that the size of $\mathcal{F}(D)$ is smaller than its expected value $\binom{k-1}{k - |D| -1}$.

Let $C\subset U$ be a set of size at most $k-2$.
We say $C$ is \textit{perfect} if the family $\mathcal{F}(C)$ is a full star on either $A$ or $B$.
Let $D\subset U$ be a set of size $k-1$.
We say $D$ is \textit{perfect} if there exists a set $F$ in $\mathcal{F}$ that contains $D$.

For every $i\in [k-1]$
let $\mathcal{P}_i$ be the collection of all perfect sets in $\binom{U}{k-i}$,
and let $\mathcal{N}_i$ be the collection of non-perfect sets in $\binom{U}{k-i}$.
Let $p_i=|\mathcal{P}_i|$ and $n_i=|\mathcal{N}_i|$ for $i \in [k-1]$ and notice that $p_i+n_i=\binom{|U|}{k-i}$.

For every $i\in \{2,\ldots, k-1\}$
let $\mathcal{P}'_i$ denote the collection of all sets $C\in \binom{U}{k-i}$ such that $C$ is contained in a perfect set in $\binom{U}{k-i+1}$,
and let $\mathcal{N}'_i$ denote the collection all of sets  $D\in \binom{U}{k-i}$ such that $D$ is not contained in any perfect set in $\binom{U}{k-i+1}$.
Let $p'_i=|\mathcal{P}'_i|$ and $n_{i}'=|\mathcal{N}'_i|$ for $i\in \{2,\ldots, k-1\}$.
Let $\mathcal{G}_i=\mathcal{N}_i\cap \mathcal{P}'_i$ and $\mathcal{B}_i=\mathcal{N}_i\cap \mathcal{N}'_i$,
and let $g_i=|\mathcal{G}_i|$ and $b_i=|\mathcal{B}_i|$  for $i\in \{2,\ldots, k-1\}$.
Let  $\mathcal{G}_1=\mathcal{N}_1$, and let $g_1=n_1$, $b_1=0$.
Note that by definition, $b_i+g_i=n_i$ and $n'_i\ge b_i$ for $i \in [k-1]$.

By the definition of perfect sets, $|\mathcal{F}(C)| = \binom{k-1}{i-1}$ for all $C\in \mathcal{P}_i$.
Later we will show that $|\mathcal{F}(C)| < \binom{k-1}{i-1}$ for all $C\in \mathcal{G}_i$.
For every $C\in \mathcal{B}_i$ it could be true that $|\mathcal{F}(C)| > \binom{k-1}{i-1}$.
However, for every $C\in \mathcal{B}_i$ there are either many sets in $\mathcal{G}_{i-1}$ containing $C$,
which means that there are many sets $D\in \binom{U}{k-i+1}$ with $|\mathcal{F}(D)|$ smaller than its expected value,
or there are many sets in $\mathcal{B}_{i-1}$, in which case we turn to consider sets in $\binom{U}{k-i+2}$ and repeat
this argument until we end up with many sets $P$ in $\binom{U}{k-1}$ with $|\mathcal{F}(P)|$ smaller than its expected value.

The next claim gives a relation between $n_i$ and $b_{i+1}$.

\begin{claim}\label{lemma:ni-bi+1}
For every $i\in [k-2]$ we have
\begin{align}
n_i\ge\frac{n-3k}{k}b_{i+1}. \notag
\end{align}
\end{claim}
\begin{proof}[Proof of Claim \ref{lemma:ni-bi+1}]
Let $C\in \mathcal{N}'_{i+1}$, and let $u\in U-C$.
By definition $C\cup\{u\}$ is a non-perfect set in $\binom{U}{k-i}$.
Therefore, we have $(k-i)n_i\ge n'_{i+1}(n-3k+i+1)\ge b_{i+1}(n-3k)$.
It follows that $n_i\ge (n-3k)b_{i+1}/k$.
\end{proof}

\begin{claim}\label{lemma:l-perfect-star-to-l-1-star}
The following statement holds for all $\ell\ge (k+1)/2$.
Suppose that $C\subset U$ is a perfect set of size $\ell$, and $\mathcal{F}(C)$ is a full star on $A$ (or on $B$) with core $v$.
Then for every $(\ell-1)$-subset $C'$ of $C$ the family $\mathcal{F}(C')$ is a star on $A$ (or on $B$) with core $v$.
\end{claim}
\begin{proof}[Proof of Claim \ref{lemma:l-perfect-star-to-l-1-star}]
Let $C \subset U$ such that $\mathcal{F}(C)$ is a full star on $A$ with core $v \in A$.
Without loss of generality we may assume that $v = a_1$.
Let $E'\in \mathcal{F}(C')$.
If $E'\subset B$, then choose a set $E$ from $\mathcal{F}(C)$, and the three sets $E\cup C, E'\cup C', B$ form a $3$-cluster in $\mathcal{F}$, a contradiction.
If $ E'\cap A\neq \emptyset$ and  $E'\cap B\neq \emptyset$, then let $x=|E'\cap A|$ and $y=|E'\cap B|$.
Since $x+y=k-\ell+1$, we have $x\le k-\ell$ and $y\le k-\ell$.
If $a_1\not\in E'\cap A$, then by the assumption that $\ell \ge (k+1)/2$ and $\mathcal{F}(C)$ is a full star, there exists a set $E\in\mathcal{F}(C)$ such that $(E'\cap A)\cap E=\emptyset$.
So the three sets $E'\cup C', E\cup C, A$ form a $3$-cluster in $\mathcal{F}$, a contradiction.
If $a_1 \in E'\cap A$, then by assumption there exists a set $E\in\mathcal{F}(C)$ such that $E'\cap A\subset E$.
However, the three sets $E\cup C, E'\cup C', B$ form a $3$-cluster in $\mathcal{F}$, a contradiction.
Therefore,  every set in $\mathcal{F}(C')$ is completely contained in $A$.

Next, we show that every set $E'\in \mathcal{F}(C')$ contains $a_1$.
Suppose there exists a set $E'\in\mathcal{F}(C')$ such that $a_1\not\in E'$.
By assumption we have $k-\ell+1+k-\ell\le k$, so there exists a set $E\in\mathcal{F}(C)$ such that $E\cap E'=\emptyset$.
However, the three sets  $E'\cup C', E\cup C, A$ form a $3$-cluster in $\mathcal{F}$, a contradiction.
Therefore, the family $\mathcal{F}(C')$ is a star on $A$ with core $a_1$.
\end{proof}

For every $i\in [k-1]$ let $w_i=\binom{k-1}{i-1}\binom{n-2k}{k-i}$ and $k_i=\binom{2k}{i}-\binom{k-1}{i-1}+1$.
Our next claim gives an upper bound for $|\mathcal{F}_i|$ for $2\le i\le (k+1)/2$.

\begin{claim}\label{lemma:upper-bound-Fi-less-k+1/2}
For every $i$ satisfying $2\le i\le (k+1)/2$ we have
\begin{align}
|\mathcal{F}_i|\le w_i+k_ib_i-n_i. \notag
\end{align}
\end{claim}
\begin{proof}[Proof of Claim \ref{lemma:upper-bound-Fi-less-k+1/2}]
Let us give an upper bound for $|\mathcal{F}(C)|$ for every $C\in \binom{U}{k-i}$.
First notice that by definition $|\mathcal{F}(C)|=\binom{k-1}{i-1}$ for all $C\in\mathcal{P}_i$.
By Claim \ref{lemma:l-perfect-star-to-l-1-star}, $|\mathcal{F}(C)|\le\binom{k-1}{i-1}-1$
for all $C\in\mathcal{G}_i$.
On the other hand, it is trivially true that $|\mathcal{F}(C)|\le\binom{2k}{i}$
for all $C\in\mathcal{B}_i$.
Therefore,
\begin{align}
|\mathcal{F}_i|
& = \sum_{C\in\mathcal{P}_i}|\mathcal{F}(C)|+\sum_{C\in\mathcal{G}_i}|\mathcal{F}(C)|+\sum_{C\in\mathcal{B}_i}|\mathcal{F}(C)| \notag\\
&\le \binom{k-1}{i-1}p_i+\left(\binom{k-1}{i-1}-1\right)g_i+\binom{2k}{i}b_i  \notag\\
&=\binom{k-1}{i-1}\binom{n-2k}{k-i}+\left(\binom{2k}{i}-\binom{k-1}{i-1}+1\right)b_i-n_i
= w_i+k_ib_i-n_i. \notag
\end{align}
Here we used that fact that $b_i+g_i=n_i$ and $n_i+p_i=\binom{n-2k}{k-i}$.
\end{proof}

Recall that Claim \ref{lemma:ni-bi+1} says that $n_i\ge (n-3k)b_{i+1}/k$.
Since $n\ge 3k \binom{2k}{k}$ and $k_{i+1} < \binom{2k}{k}$, we have $n_i/2 \ge k_{i+1}b_{i+1}$.
Combining this inequality with Claim \ref{lemma:upper-bound-Fi-less-k+1/2} we obtain the following claim.

\begin{claim}\label{lemma:upper-bound-Fi-less-k+1/2-simple}
For every $\ell$ satisfying  $1\le \ell \le (k+1)/2$ we have
\begin{align}
\sum_{i=0}^{\ell}|\mathcal{F}_i|\le \sum_{i=1}^{\ell}w_i-\sum_{i=1}^{\ell}\frac{n_i}{2}. \notag
\end{align}
\end{claim}
\begin{proof}[Proof of Claim \ref{lemma:upper-bound-Fi-less-k+1/2-simple}]
The case $\ell=1$ follows from the inequality that
\begin{align}
|\mathcal{F}_0|+|\mathcal{F}_1|\le|\mathcal{M}|=\binom{n-2k}{k-1}-n_1. \notag
\end{align}
For $\ell\ge 2$ by Claim \ref{lemma:upper-bound-Fi-less-k+1/2} we obtain
\begin{align}
\sum_{i=0}^{\ell}|\mathcal{F}_i|\le \sum_{i=1}^{\ell}(w_i+k_ib_i-n_i)
=\sum_{i=1}^{\ell}w_i-\sum_{i=1}^{\ell-1}(n_i-k_{i+1}b_{i+1})-n_{\ell} \le  \sum_{i=1}^{\ell}w_i-\sum_{i=1}^{\ell}\frac{n_i}{2}. \notag
\end{align}
\end{proof}

The next step is to extend Claim \ref{lemma:upper-bound-Fi-less-k+1/2-simple} to all $\ell>(k+1)/2$.

\begin{claim}\label{lemma:C-P-perfect-set-implies-star}
Let $C\subset U$ be a set of size $\ell\ge 2$.
Suppose that $\mathcal{F}(C)$ is a full-star on $A$ (or on $B$) with core $v$
and there exists a perfect set $P\in\binom{U}{k-1}$ containing $C$.
Then, for every $(\ell-1)$-subset $C'\subset C$ the family
$\mathcal{F}(C')$ is a star on $A$ (or on $B$) with core $v$.
\end{claim}
\begin{proof}[Proof of Claim \ref{lemma:C-P-perfect-set-implies-star}]
Let $C \subset U$ be a set of size $\ell$ such that  $\mathcal{F}(C)$ is a full-star on $A$ with core $v$.
Without loss of generality we may assume that $v = a_{1}$.
Let $P \in \binom{U}{k-1}$ be a perfect set containing $C$.
By the definition of perfect set there exists a set $F\in\mathcal{F}$ containing $P$.
Suppose that $F=P\cup \{u\}$, and we want to show that $u=a_1$.
Suppose that $u\not\in A$.
Then for every $E\in\mathcal{F}(C)$ the three sets $A,  F, E\cup C$  form a $3$-cluster in $\mathcal{F}$, a contradiction.
Therefore, $u\in A$.

Now suppose for the contrary that $u \neq a_1$.
Then by assumption there exists a set $E\in\mathcal{F}(C)$ not containing $u$ and, hence,
the three sets $A, F, E\cup C$ form a $3$-cluster in $\mathcal{F}$, a contradiction.
Therefore, $u=a_1$.

Let $C'\subset C$ be a set of size $\ell-1$ and $E' \in \mathcal{F}(C')$.
If $E'\subset B$, then for every $E\in \mathcal{F}(C)$ the three sets $E\cup C, E'\cup C', B$ form a $3$-cluster in $\mathcal{F}$, a contradiction.
If $ E'\cap A\neq \emptyset$ and  $E'\cap B\neq \emptyset$, then let $x=|E'\cap A|$ and $y=|E'\cap B|$.
Since $x+y=k-\ell+1$, we have $x\le k-\ell$ and $y\le k-\ell$.
If $x\le k-\ell-1$, then by assumption there exists a set $E\in\mathcal{F}(C)$ containing $E'\cap A$.
However, the three sets $E\cup C, E'\cup C', B$ form a $3$-cluster in $\mathcal{F}$, a contradiction.
Therefore, we may assume that $x=k-\ell$.
If $a_1\in E'\cap A$, then there exists a set $E\in\mathcal{F}(C)$ such that $E'\cap A= E$.
However, the three sets $E\cup C, E'\cup C', B$ form  a $3$-cluster in $\mathcal{F}$, a contradiction.
If $a_1 \not\in E'\cap A$, then the three sets $A, F, E'\cup C'$ form a $3$-cluster in $\mathcal{F}$, a contradiction.
Therefore, every set in $\mathcal{F}(C')$ is completely contained in $A$.

Suppose that there is a set $E'\in\mathcal{F}(C')$ not containing $a_1$,
then the three sets $A, F, E'\cup C'$ would form a $3$-cluster in $\mathcal{F}$, a contradiction.
Therefore, every set in $\mathcal{F}(C')$ contains $a_1$, and this complete the proof of Claim \ref{lemma:C-P-perfect-set-implies-star}.
\end{proof}

Let $c=\lf(k+1)/2\rf$ and let $m=\lf k/2 \rf$, and notice that $m+c=k$.
The next claim shows that $(\ref{inequality-main-goal})$ holds for $\ell = c+1$.

\begin{claim}\label{lemma:inequality-sum-F-i-to-c+1}
We have
\begin{align}
\sum_{i=0}^{c+1}|\mathcal{F}_i|\le \sum_{i=1}^{c+1}w_i-\sum_{i=1}^{c+1}\frac{n_i}{4}. \notag
\end{align}
\end{claim}
\begin{proof}[Proof of Claim \ref{lemma:inequality-sum-F-i-to-c+1}]
Similar to the proof of Claim \ref{lemma:upper-bound-Fi-less-k+1/2},
for every $C\in\mathcal{P}_{c+1}$ we have $|\mathcal{F}(C)|=\binom{k-1}{c}$,
and for every $C\in\mathcal{B}_{c+1}$ we have $|\mathcal{F}(C)|\le\binom{2k}{c+1}$.

For every perfect set $D\in \binom{U}{m}$ we say that $D$ is a \textit{good container} if $D$ itself is contained in a perfect $(k-1)$-set,
otherwise we say that $D$ is a \textit{bad container}.
Let $\mathcal{S}$ be the collection of all sets in $\mathcal{G}_{c+1}$ that are contained in a good container.
Let $\mathcal{T}$ be the collection of all sets in $\mathcal{G}_{c+1}$ that are not contained in any good container.
Let $s=|\mathcal{S}|$ and $t=|\mathcal{T}|$.
Since every bad container in $\binom{U}{m}$ has $m$ subsets of size $m-1$,
the number of bad containers in $\binom{U}{m}$ is at least $t/m$.

Let $D\in \binom{U}{m}$ be a bad container.
Then for every $E\in \binom{U-D}{k-m-1}$ the set $D\cup E$ is non-perfect in $\binom{U}{k-1}$.
Therefore,  $n_1\ge {\binom{n-2k-m}{c-1}}t/\left( m\binom{k-1}{m} \right)$.
By definition, every set $C\in\mathcal{G}_{c+1}$ is contained in a perfect set $D\in\binom{U}{m}$.
If $C\in \mathcal{S}$, then by Claim \ref{lemma:C-P-perfect-set-implies-star}, $|\mathcal{F}(C)|\le \binom{k-1}{c}-1$.
If $C \in \mathcal{T}$, then it is trivially true that $|\mathcal{F}(C)|\le \binom{k}{c+1}$.
Therefore,
\begin{align}
|\mathcal{F}_{c+1}|
& = \sum_{C\in\mathcal{P}_{c+1}}|\mathcal{F}(C)|
	+\sum_{C\in\mathcal{B}_{c+1}}|\mathcal{F}(C)|
	+\sum_{C\in\mathcal{S}}|\mathcal{F}(C)|
	+\sum_{C\in\mathcal{T}}|\mathcal{F}(C)|  \notag\\
& \le \binom{k-1}{c}p_{c+1}+\binom{2k}{c+1}b_{c+1}+\left(\binom{k-1}{c}-1\right)s+\binom{2k}{c+1}t \notag \\
& =w_{c+1}+k_{c+1}b_{c+1}+k_{c+1}t-n_{c+1}. \notag
\end{align}
Here we used the fact that $s+t=g_{c+1}$, $g_{c+1}+b_{c+1}=n_{c+1}$ and $n_{c+1}+p_{c+1}=\binom{n-2k}{k-c-1}$.
Combining the inequality above with Claim \ref{lemma:upper-bound-Fi-less-k+1/2}, we obtain
\begin{align}
\sum_{i=0}^{c+1}|\mathcal{F}_i|\le \sum_{i=1}^{c+1}(w_i+k_ib_i-n_i)+k_{c+1}t. \notag
\end{align}
Since  $n_1/4\ge k_{c+1}t$ and $n_i/2\ge k_{i+1}b_{i+1}$,
\begin{align}
\sum_{i=0}^{c+1}|\mathcal{F}_i|\le \sum_{i=1}^{c+1}w_i-\sum_{i=1}^{c+1}\frac{n_i}{4}. \notag
\end{align}
\end{proof}

\begin{claim}\label{claim-C-is-contained-in-perfec-set}
Every set $C\subset U$ of size at most $k-c$ is contained in a perfect $(k-1)$-set.
\end{claim}
\begin{proof}[Proof of Claim \ref{claim-C-is-contained-in-perfec-set}]
Let $C\subset U$ be a set of size $\ell\le k-c$.
Suppose that $C$ is not contained in any perfect $(k-1)$-set.
Then for every $S\in \binom{U-C}{k-\ell-1}$ the set $C\cup S$ is non-perfect and of size $k-1$.
Therefore, we have $n_1\ge \binom{n-2k-\ell}{k-\ell-1}/\binom{k-1}{\ell}\ge \binom{n-2k-\ell}{c-1}/\binom{k-1}{\ell}$.
On the other hand, we have  $\sum_{i=c+2}^{k-1}|\mathcal{F}_i|\le \sum_{i=c+2}^{k-1}\binom{2k}{i}\binom{n-2k}{k-i}$.
Since $n\ge 3k\binom{2k}{k}$, $n_1/4> \sum_{i=c+2}^{k-1}|\mathcal{F}_i|$.
Therefore,
by Claim \ref{lemma:inequality-sum-F-i-to-c+1},
\begin{align}
\sum_{i=0}^{k-1}|\mathcal{F}_i|
=\sum_{i=1}^{c+1}|\mathcal{F}_i|
+\sum_{i=c+2}^{k-1}|\mathcal{F}_i|
\le \sum_{i=1}^{c+1}w_i-\sum_{i=1}^{c+1}\frac{n_i}{4}
+\sum_{i=c+2}^{k-1}\binom{2k}{i}\binom{n-2k}{k-c-2}
<\sum_{i=1}^{k-1}w_i, \notag
\end{align}
and we are done.
So we may assume that $C$ is  contained in a perfect $(k-1)$-set.
\end{proof}

\begin{claim}\label{claim-Fi-inequality-for-all-i}
The inequality $|\mathcal{F}_i|\le w_i+t_ib_i-n_i$ holds for all $i\ge c+1$.
\end{claim}
\begin{proof}
By Claim \ref{claim-C-is-contained-in-perfec-set},
every set $C\subset U$ of size at most $k-c$ is contained in a perfect $(k-1)$-set.
Therefore, by Claim \ref{lemma:C-P-perfect-set-implies-star},
\begin{align}
|\mathcal{F}_i|
& =\sum_{C\in \mathcal{P}_{i}}|\mathcal{F}(C)|+\sum_{C\in \mathcal{G}_{i}}|\mathcal{F}(C)|+\sum_{C\in \mathcal{B}_{i}}|\mathcal{F}(C)| \notag\\
&\le \binom{k-1}{i-1}p_i+\left(\binom{k-1}{i-1}-1\right)g_i+\binom{2k}{i}b_i. \notag
\end{align}
\end{proof}

By Claims \ref{lemma:ni-bi+1}, \ref{lemma:upper-bound-Fi-less-k+1/2}, and \ref{claim-Fi-inequality-for-all-i},
\begin{align}
\sum_{i=0}^{k-1}|\mathcal{F}_i|
& \le \sum_{i=1}^{k-1}(w_i+t_ib_i-n_i)=\sum_{i=1}^{k-1}w_i-\sum_{i=1}^{k-2}(n_i-t_{i+1}b_{i+1})-n_{k-1} \notag\\
&\le \sum_{i=1}^{k-1}w_i-\sum_{i=1}^{k-1}\frac{n_i}{2}, \notag
\end{align}
which proves $(\ref{inequality-main-goal})$,
and equality holds if and only if $C$ is perfect for every $C\in \binom{U}{i}$ and for every $i\in[k-1]$,
which implies that $\mathcal{F}$ is the disjoint union of a $k$-set and a full star.
\end{proof}


\section{Acknowledgement}
We are very grateful to Dhruv Mubayi for his guidance, expertise, fruitful discussions that greatly assisted this research,
and suggestions that greatly improved the presentation of this paper.
We are also very grateful to the referee for a careful reading of this manuscript and several helpful suggestions.
\bibliographystyle{abbrv}
\bibliography{structureintersecting}
\end{document}